\documentclass[12pt,a4paper]{article}

\usepackage{graphicx}                
\usepackage{color,xcolor}            
\usepackage{hyperref}
\usepackage{enumerate}

\usepackage{float}
\usepackage{hyperref}
\usepackage{graphicx}
\usepackage{amsthm}  
\usepackage{amsmath}
\usepackage{mathrsfs}                
\usepackage{amsfonts}
\usepackage{amsbsy}
\usepackage{amssymb}   
\usepackage{graphics}
\usepackage{color}		
\usepackage{epsfig}
\usepackage{mathrsfs}
\usepackage{bm}                      
\usepackage{bbding,manfnt}           
\usepackage{sidecap}
\usepackage{caption}
\usepackage{subcaption}
\usepackage{cases}
\usepackage{esint}
\usepackage{enumerate}
\usepackage{lettrine}                
\usepackage[top=1in,bottom=1in,left=1.2in,right=1.2in]{geometry}

\theoremstyle{plain}
\newtheorem{theorem}{Theorem}[section]
\newtheorem{lemma}[theorem]{Lemma}
\newtheorem{proposition}[theorem]{\bf Proposition}

\theoremstyle{remark}

\theoremstyle{definition}
\newtheorem{definition}[theorem]{\bf Definition}
\newtheorem{remark}[theorem]{\bf Remark}

\numberwithin{equation}{section}

\newcommand*{\R}{\ensuremath{\mathbb{R}}}

\newcommand*{\oo}{\ensuremath{C}}
\renewcommand*{\L}{\ensuremath{\mathfrak{L}}}

\newcommand*{\D}{\ensuremath{\mathbb{D}}}
\newcommand{\C}[1]{\mathbf{C^{#1}}}
\newcommand{\lnorm}[1]{\ensuremath{\mathbb{L}}^{#1}}
\newcommand{\brz}{B_r(0)}

\newcommand{\la}{\ensuremath{\lambda}}

\newcommand*{\ul}{\ensuremath{u^\mathrm{L}}}

\newcommand*{\ur}{\ensuremath{u^\mathrm{R}}}
\newcommand*{\unu}{\ensuremath{u^\nu}}

\newcommand*{\e}{\ensuremath{\varepsilon}}

\newcommand{\ue}{\ensuremath{u^{\varepsilon}}}

\newcommand*{\iu}{\ensuremath{\bar u}}
\newcommand*{\iv}{\ensuremath{\bar v}}

\newcommand{\enu}{\varepsilon^{\nu}}
\newcommand{\tu}{\tilde u}

\newcommand*{\go}{\ensuremath{g_1}}
\newcommand*{\gt}{\ensuremath{g_2}}

\newcommand*{\got}{\ensuremath{g_1(t)}}
\newcommand*{\gtt}{\ensuremath{g_2(t)}}

\newcommand*{\ith}{\ensuremath{i\text{-th}}}

\renewcommand*{\P}{\mathcal{P}}
\newcommand*{\NP}{\mathcal{N}\mathcal{P}}
\newcommand{\xat}{\ensuremath{x_\alpha(t)}}

\newcommand{\lug}{\Lambda(\iu,\go,\gt)}
\newcommand{\xa}{x_{\alpha}}

\DeclareMathOperator*{\tvv}{\mathrm{Tot.Var.}}
\DeclareMathOperator*{\rank}{\mathrm{rank}}
\def\rn#1{\mathbb{R}^{#1}}
\newcommand{\tv}[2]{\ensuremath{\tvv_{#2}{\left(#1\right)}}}

\newcommand*{\pt}{\ensuremath{\partial_t}}      
\newcommand*{\px}{\ensuremath{\partial_x}}

\title{Local exact boundary controllability of entropy solutions to a class of hyperbolic systems of conservation laws}
\author{Tatsien Li and Lei Yu}
\date{}

\begin{document}
\maketitle

\begin{abstract}
In this paper, we study the local exact boundary controllability of entropy solutions to a class linearly degenerate hyperbolic systems of conservation laws with constant multiplicity. The authors prove the two-sided boundary controllability, one-sided boundary controllability and two-sided controllability with less controls, by applying the strategy used in \cite{Li_controllability-book} originally for classical solutions with essential modifications. Our constructive method is based on the well-posedness of semi-global solutions constructed by the limit of $ \e $-approximate front tracking solutions to the mixed initial-boundary value problem with general nonlinear boundary conditions and some further properties on both $ \e $-approximate front tracking solutions and entropy solutions, which are obtained in \cite{Li_Yu-onesided2016} and \cite{Yu-BCLD}.
\end{abstract}

\vspace{12pt}

	
\section{Introduction}
In this paper, we study the local exact boundary controllability for $n\times n$ hyperbolic system of conservation laws in one space dimension:
\begin{equation}
\label{e:intro}
\pt H(u)+\px G(u)=0,\qquad t\ge 0,\ 0<x<L,
\end{equation}
where $u$ is an $n$-vector valued unknown function of $ (t,x) $, $G$ and $ H$ are smooth $n$-vector valued functions of $ u $, defined on a ball $B_r(0)$ centered at the origin in $\rn n$ with suitable small radius $r$.	
\subsection{Preliminary assumptions and definitions}

	
For the system \eqref{e:intro}, we have the following hypothesis:

$\textbf{(H1)}$ System \eqref{e:intro} is hyperbolic, that is, for any given $u\in B_r(0)$, the matrix $DH(u)$ is non-singular and the matrix $(DH(u))^{-1}DG(u)$ has $n$ real eigenvalues $  \la_i(u)$ $ (i=1,...,n) $ and there exist a complete set of left (resp. right) eigenvectors $ \{ l_1(u),...,l_n(u)\} $ (resp. $ \{r_1(u),...,r_n(u)\} $).

$\textbf{(H2)}$ For any $ u\in \brz $, each eigenvalue of $(DH(u))^{-1}DG(u)$ has a constant multiplicity. To fix the idea, we suppose that
\begin{equation*}
\label{h:ev}
\la_1(u)< \cdots<\la_{k}(u)<\la_{k+1}(u)\equiv\cdots\equiv\la_{k+p}(u)<\la_{k+p+1}(u)<\cdots<\la_n(u),
\end{equation*}
where $ \la(u):=\la_{k+1}(u)\equiv\cdots\equiv \la_{k+p}(u) $ is an eigenvalue with constant multiplicity $ p\ge 1 $. When $ p=1 $, the system \eqref{e:intro} is \emph{strictly hyperbolic}.

$\textbf{(H3)}$ There are no zero eigenvalues, that is, there exist an $m\in \{1,...,n\}$ and a constant $c>0$, such that
\begin{equation}
\label{h:nonzero-char}
\lambda_m(u)<-c<0<c<\lambda_{m+1}(u), \qquad \forall u \in \brz.
\end{equation}
Under this assumption $ DG(u) $ is also a non-singular matrix. Without loss of generality, we assume that $ 1\le k<\cdots <k+p\le m $, i.e. the eigenvalue $ \la(u) $ is negative. The other situation is similar.

$\textbf{(H4)}$ All eigenvalue $ \la_i $ $ (1\le i\le n) $ are linear degenerate in the sense of Lax \cite{Lax1987}. Recall that  the $\ith$ eigenvalue is linearly degenerate if
\begin{equation*}
\label{h:LD}
D \lambda_i(u)\cdot r_i(u) \equiv 0,\qquad \forall u\in \brz,
\end{equation*}
In fact, the eigenvalue $ \la(u) $ with constant multiplicity $ p\ge 2 $ must be linearly degenerate (see Lemma \ref{l:ecm}).

$\textbf{(H5)}$ Assume that the system \eqref{e:intro} possesses a convex entropy $ \eta(u) $ together with an entropy flux $ \zeta(u) $. Recall that a continuously differentiable convex function $\eta(u):\rn n\to \R$ is called a \emph{convex entropy} of system \eqref{e:intro}, with an \emph{entropy flux} $\zeta(u):\rn n\to \R$, if we have
\begin{equation}
\label{d:nc-entropy}
\begin{split}
D\eta(u)(DH(u))^{-1} DG(u) = D \zeta(u).
\end{split}
\end{equation}

\vspace{12pt}

By \eqref{h:nonzero-char}, the boundary $ x=0 $ and $ x=L $ are non-characteristic. We prescribe the following general nonlinear boundary conditions:
\begin{align}
x=0:\ b_1(u)=\go(t), \label{bc:intro-a}\\
x=L:\ b_2(u)=\gt(t),
\end{align}
where $g_1:\R^+\to \rn{n-m}$, $g_2:\R^+\to \rn{m}$ are given boundary data functions and $b_1 \in \C1(B_r(0)$; $\rn{n-m}),\ b_2 \in \C1(B_r(0); \rn{m})$. In order to guarantee the well-posedness for the forward mixed initial-boundary value problem of system \eqref{e:intro}, we assume that

$\textbf{(H6)}$ $ b_1 $ and $ b_2 $ satisfy the following conditions, respectively (see \cite{Li-Yu_boundary-value}):
\begin{equation}
\label{h:bc}
\begin{split}
& \det \left[
Db_1(u) \cdot r_{m+1}(u)\ |\cdots |\ Db_1(u) \cdot r_{n}(u)
\right] \neq 0, \\
& \det \left[
Db_2(u) \cdot r_{1}(u)\ | \cdots |\ Db_2(u) \cdot r_{m}(u)
\right] \neq 0,
\end{split}
\quad \forall u\in \brz.
\end{equation}
Here the value of $u(t,0)$ and $ u(t,L)$ should be understood as the inner trace of the function $u(t,x)$ on the boundary $x=0$ and $ x=L$, respectively.

\vspace{12pt}

Now we can write the mixed initial-boundary value problem of \eqref{e:intro} as follows:
\begin{equation}
\label{eqt:ibvp}
\begin{cases}
\pt H(u)+\px G(u)=0,& t\ge 0,\ 0<x<L,\\
t=0:\ u=\iu(x),& 0<x<L,\\
x=0:\ b_1(u)=\go(t),& t>0,\\
x=L:\ b_2(u)=\gt(t),& t>0.
\end{cases}
\end{equation}

Before giving the definition of entropy solution,
\begin{definition}
\label{d:entropy-solution}	
For any given $T>0$, $u=u(t,x)\in \lnorm{1}((0,T)\times (0,L))$ is an \emph{entropy solution to system \eqref{e:intro}} on the domain $\D:=\{\ 0< t< T,\ 0<x<L\}$ if
\begin{enumerate}[(1)]
\item $u$ is a weak solution to \eqref{e:intro} on the domain $ \D $ in the sense of distributions, that is, for every $\phi \in C_c^1(\D)$ we have
\[
\int^T_0\int^L_0\pt\phi(t,x)H(u(t,x))+\px\phi(t,x)G(u(t,x))dxdt=0.
\]
\item $u$ is entropy admissible in the sense that there exists a convex entropy $\eta(u)$ with entropy flux $ q(u)$ for system \eqref{e:intro}, such that for every non-negative function $\phi \in C_c^1(\D)$ we have
\begin{equation}
\label{d:entropy}
\int^T_0\int^L_0\pt\phi(t,x) \eta(u(t,x))+\px\phi(t,x) q(u(t,x))dxdt\geq 0.
\end{equation}
\end{enumerate}
	
Moreover, if $ u $ also satisfies the following initial-boundary conditions:
	
\noindent (3) $ \ $ for a.e. $x\in (0,L)$, $ \lim\limits_{t\to0+}u(t,x)=\iu(x)$ and
\begin{equation*}
\lim_{x\to 0+}b_1(u(t,x))=\got,\qquad \lim_{x\to L-}b_2(u(t,x))=\gtt,\qquad \text{a.e. $t \in (0,T)$},
\end{equation*}
then we say that $ u $ is an \emph{entropy solution to the mixed initial-boundary value problem} \eqref{eqt:ibvp} on the domain $\D$.
\end{definition}

\subsection{Previous studies and main result}

Roughly speaking, the exact boundary controllability for the system  \eqref{e:intro} requires us to consider the following question: For any two given admissible initial and final state, is it possible to find suitable boundary condition as a control, such that the solution to the corresponding mixed initial-boundary value problem \eqref{eqt:ibvp} reaches the desired final state in finite time.

Most results for boundary controllability of quasilinear hyperbolic system have been obtained in the framework of classical solutions. Recently, Li and Rao systematically studied the boundary controllability for general quasilinear hyperbolic system, based on well-posedness of semi-global $ C^1 $ solutions. They proved an exact local boundary controllability, driving any given small initial data to any small final data (see \cite{Li_controllability-book} and reference therein).

It is well known that the classical solutions of quasilinear hyperbolic systems usually blow up in a finite time even though the initial data is sufficiently smooth. Thus it is natural to consider weak solutions containing shocks which concern important physical phenomenon. In \cite{Kong-control2002}-\cite{Kong-control2005}, the authors proved the global exact boundary controllability of a class of hyperbolic systems of conservation laws with linearly degenerate characteristic families in the framework of piecewise $ C^1 $ solutions. While, concerning more general entropy solutions, the study of boundary controllability of nonlinear hyperbolic systems is still vastly open. 
So far for the system case, there are only results concerning special models of hyperbolic conservation laws, for example Temple system \cite{AnconaCoclite2005} and Euler equations \cite{Glass2007}-\cite{Glass2014}.

Recently, in \cite{Li_Yu-onesided2016} and \cite{Yu-BCLD}, we proved the one-sided exact boundary null controllability of entropy solutions is studied for a class of general hyperbolic systems of conservation laws satisfying Hypothesis (H1)-(H3) and the assumption that all negative (or positive) characteristic families are linearly degenerate, by means of the constructive method used in \cite{Li_controllability-book} originally for the local exact boundary controllability in the framework of classical solutions with essential modifications.

In the present paper, following the same strategy used in \cite{Li_Yu-onesided2016} and \cite{Yu-BCLD}, we study the local exact boundary controllability of entropy solutions to a class of general linearly degenerate hyperbolic system of conservation laws \eqref{e:intro} with eigenvalue of constant multiplicity.

The main results of this paper is the following three theorems.
\begin{theorem}[\bf Two-sided boundary control]
\label{t:twoside-cl}
Let system \eqref{e:intro} and $ b_1,\ b_2 $ satisfy hypotheses (H1)-(H6). If
\begin{equation}
\label{d:tcl-two}
T>L \max\left\{\frac{1}{|\lambda_m(0)|},\frac{1}{\lambda_{m+1}(0)}\right\},
\end{equation}
then, for any given initial data $\iu $ and final data $ u_1 $ with $\displaystyle\tv{\iu}{0<x<L}+|\iu(0+)|$ and $\displaystyle\tv{u_1}{0<x<L}+|u_1(0+)|$ sufficiently small, there exist boundary controls $ g_1 $ and  $g_2$ with $\displaystyle\tv{g_1}{0<t<T}+|g_1(0+)|$ and $\displaystyle\tv{g_2}{0<t<T}+|g_2(0+)|$ sufficiently small, such that the mixed initial-boundary value problem \eqref{eqt:ibvp} admits an entropy solution $ u=u(t,x) $ on the domain $\{\ 0<t<T,0<x<L\}$, satisfying the final condition
\begin{equation}
\label{e:fc}
t=T:\quad u=u_1, \qquad \forall x\in (0,L).
\end{equation}
\end{theorem}
	
\begin{theorem}[\bf One-sided boundary control]
\label{t:oneside-cl}
Let system \eqref{e:intro} and $ b_1,\ b_2 $ satisfy hypotheses (H1)-(H6). Suppose further that
\begin{equation}
\label{h:os}
\bar m:=n-m\le m
\end{equation}
and
\begin{equation}
\label{h:os-b_1}
\rank \big( \left[
Db_1(u) \cdot r_{1}(u)\ |\cdots |\ Db_1(u) \cdot r_{m}(u)
\right] \big)= \bar m,\quad \forall u\in \brz.
\end{equation}
If
\begin{equation}
\label{d:tcl-one}
T>L \left\{\frac{1}{|\lambda_m(0)|}+\frac{1}{\lambda_{m+1}(0)}\right\},
\end{equation}
then, for any given initial data $\iu \in$ and final data $ u_1$ with $\displaystyle\tv{\iu}{0<x<L}+|\iu(0+)|$ and $\displaystyle\tv{u_1}{0<x<L}+|u_1(0+)|$ sufficiently small, and for any given boundary data $ g_1$ at $ x=0 $ with $\displaystyle\tv{g_1}{0<t<T}+|g_1(0+)|$ sufficiently small, there exists a boundary control $g_2$, acting on the boundary $x=L$, such that the mixed initial-boundary value problem \eqref{eqt:ibvp}
admits an entropy solution $ u=u(t,x) $ on the domain $\{\ 0<t<T,0<x<L\}$, satisfying the final condition \eqref{e:fc}.
\end{theorem}

\begin{theorem}[\bf Two-sided boundary control with less controls]
\label{t:ts-lcl}
Let system \eqref{e:intro} and $ b_1,\ b_2 $ satisfy hypotheses (H1)-(H6). Suppose further that \eqref{h:os} holds.
Letting $ \tilde{b}_2:\brz\to \rn{\bar m} $ be the vector-value function consists of the first $ \bar m $ components of $ b_2 $, without loss of generality, we suppose that
\begin{equation*}
\label{h:bc-less}
\rank \big( \left[
D\tilde b_2(u) \cdot r_{m+1}(u)\ |\cdots |\ D\tilde b_2(u) \cdot r_{\bar n}(u)\right] \big)= \bar m,\quad \forall u\in \brz.
\end{equation*}
If $ T>0 $ satisfies \eqref{d:tcl-one}, then, for any given initial data $\iu$ and final data $ u_1$ with $\displaystyle\tv{\iu}{0<x<L}+|\iu(0+)|$ and $\displaystyle\tv{u_1}{0<x<L}+|u_1(0+)|$ sufficiently small, and for any given part of boundary data $ \tilde g_2:(0,T)\to \rn{\bar m} $ with $\displaystyle\tv{\tilde g_2}{0<t<T}+|\tilde g_2(0+)|$ sufficiently small, there exists boundary control $g_1$ at $x=0$ and boundary control $\hat g_2:(0,T)\to \rn{m-\bar m} $, such that the mixed initial-boundary value problem \eqref{eqt:ibvp} associated with
\begin{equation*}
g_2=
\begin{pmatrix}
\tilde{g}_2\\
\hat g_2
\end{pmatrix}
\end{equation*}
admits an entropy solution $ u=u(t,x) $ on the domain $\{\ 0<t<T,0<x<L\}$, satisfying the final condition \eqref{e:fc}.
\end{theorem}

\begin{remark}
In \cite{Li_Yu-onesided2016}, we have proved the one-sided boundary null controllability for a wilder class of hyperbolic conservation laws and give the sharp time of control. Although the corresponding two-sided boundary null controllability can be obtained as a corollary, the estimate of time for controllability is no longer optimal. In this paper, by further assumption that all eigenvalue are linearly degenerate, we obtain systematically the exact boundary controllability for the system \eqref{e:intro}, where the final state can be any small BV function close to the equilibrium, and we obtain the estimate \eqref{d:tcl-two} and \eqref{d:tcl-one} on the time for the two-sided boundary control and the one-sided boundary control (or two-sided boundary with less control) are sharp, respectively. While our results as that in \cite{Li_Yu-onesided2016} still have advantages such as the number  $ n $ of equations in the system \eqref{e:intro} can be any integer $ \ge 1 $ and the general nonlinear boundary conditions are taken into account.
\end{remark}


\subsection{Main ideas of proof and structure of paper}
\label{s:structure}
As in \cite{Li_Yu-onesided2016}, throughout this paper, by the \emph{solution} to a mixed initial-boundary value problem for the system \eqref{eqt:ibvp}, we mean the limit of a convergent sequence of corresponding $ \e $-approximate front tracking solutions (see Definition \ref{def:epsilon-sol}). This kind of solution is actually an entropy solution, provided that the system possesses a convex entropy.


Our treatment follows closely to our paper \cite{Li_Yu-onesided2016}, which concerns a class of hyperbolic systems of conservation laws with the assumption that all negative (or positive) characteristic families are linearly degenerate. Under this essential assumption, we can obtain the equivalence between the solution to the forward problem \eqref{e:intro} and the rightward problem
\begin{equation}
\label{e:right-intro}
\px G(u)+\pt H(u)=0,
\end{equation}
where $ x $ is regarded as the ``time" variable and $ t $ as the ``space" variable. And vice versa.

In this paper, under stronger assumption that all characteristic families are linear degenerate, we can now obtain analogous boundary controllability as for the classical solutions case treated in \cite{Li_controllability-book}. The key idea is that now the entropy inequality \eqref{d:entropy} is actually an equality. Then we can solve backward the system \eqref{e:intro} in the same way as for the forward case and the solutions are equivalent in both senses.

\vspace{6pt}
The paper is organized as follows. In Section \ref{s:semi-global} we recall the results of well-posedness of semi-global solutions as the limits of $ \e $-approximate solutions, which are mainly proved in \cite{Li_Yu-onesided2016} and \cite{Yu-BCLD}. In Section \ref{s:lebc}, we give the proofs of Theorem \ref{t:twoside-cl}-\ref{t:ts-lcl} following the main strategy in \cite{Li_controllability-book}.

\section{Semi-global solutions}
\label{s:semi-global}
In this section, we collect results about the well-posedness of semi-global solutions to the mixed initial-boundary value problem \eqref{eqt:ibvp}, which were proved in \cite{Li_Yu-onesided2016} and \cite{Yu-BCLD}. In fact, all results (except some in Section  \ref{s:further-properties}) hold for more general systems whose characteristic families are either genuinely nonlinear or linearly degenerate.

Throughout this paper, in order to avoid abusively using constants, we denote by the notation $\oo$ a positive constant which depends only on system \eqref{e:intro}, constant $ L $ and functions $b_1,b_2$, but is independent of the special choice of initial data $ \iu $, boundary data $
\go,\gt $ and time $T$. Moreover, we denote by $ C(T) $ a positive constant which depends also on time $ T $.

\subsection{Preliminaries}
\label{s:preliminary}
For the system \eqref{e:intro}, we normalize the left and right eigenvectors $l_i(u)$ and $ r_i(u)$ $ (i=1,...,n) $ of $(DH)^{-1}DG(u)$, so that
\[
l_i(u) \cdot r_j(u) \equiv\delta_{ij},\quad i,j =1,...,n,
\]
where $ \delta_{ij} $ is the Kronecker symbol.

For any given $u\in \brz$, when the $ \it $ eigenvalue $ \la_i(u) $ is simple,  let $\sigma\mapsto R_i(\sigma)[u]$ denote the $i$-rarefaction curve passing through $u$ 
and let $\sigma\mapsto S_i(\sigma)[u]$ denote the $i$-shock curve passing through $u$. 
If $ \la_i(u) $ is linearly degenerate, we know that the $ i$-rarefaction curve and $ i $-shock curve coincide. 

For the eigenvalue $ \la(u) $ with constant multiplicity $ p\ge 2 $, one has the following
\begin{lemma}[see \cite{Frei-LD}]
\label{l:ecm}
The eigenvalue $ \la(u) $ with constant multiplicity $ p\ge 2 $ must be linearly degenerate, that is
\[
\nabla \la(u) \cdot r_{j}(u)\equiv 0 \quad (j=k+1,...,k+p), \quad \forall u\in \brz.
\]
Moreover, for any $ u^-\in \brz $, there exists a $ p $-dimensional connected smooth manifold $\Sigma(u^-)$ in a neighborhood of $ u^- $ with $ u^-\in \Sigma(u^-) $, where $ \Sigma(u^-) $ can be expressed by the following smooth parametric representation
\[
u=\Psi_{k+1}(\sigma_{k+p},...,\sigma_{k+1})[u^-], \quad \sigma_{j}\in [-\sigma_0,\sigma_0] \quad (j=k+1,...k+p),
\]
for some small $ \sigma_0 $, such that
\[
{\partial \over \partial \sigma_{j}}u(0,..,0)[u^-]=r_{j},\quad (j=k+1,...,k+p).
\]
In other words, for any $ u^+\in \Sigma(u^-)$, there exist uniquely small numbers $ \sigma_{k+1},...,\sigma_{k+p}$ such that $u=\Psi_{k+1}(\sigma_{k+p},...,\sigma_{k+1})[u^-]  $, and any discontinuity associate with the eigenvalue $ \la(u) $
\begin{equation}
\label{e:cd}
u_{k+1}=
\begin{cases}
u^+, & x>st,\\
u^-, &x<st
\end{cases}
\end{equation}
is always a contact discontinuity, i.e. we have
\begin{equation}
\label{condition-1}
\begin{cases}
G(u^+)-G(u^-)=s (H(u^+)-H(u^-),\\
s=\la(u^-)=\la(u^-)
\end{cases}
\end{equation}
On the other hand, if $ u^+ $ is sufficiently close to $ u^- $, then the solution is a contact discontinuity implies that
\begin{equation}
\label{condition-2}
\begin{cases}
u^+\in\Gamma(u^-),\\
s=\la(u^-),
\end{cases}
\end{equation}
which means that on this contact discontinuity \eqref{e:cd}, condition \eqref{condition-1} is equivalent to condition \eqref{condition-2}.

\end{lemma}

\vspace{12pt}

\subsection{Solutions as the limit of $\e$-approximate front tracking solutions}
\label{s:entropy-sol}
We first give the definition of $ \e $-approximate front tracking solutions, which is the same as the one given in \cite{Li_Yu-onesided2016} but modified for the linear degenerate case.
\begin{definition}
\label{def:epsilon-sol}
For any given time $ T>0 $ and any fixed $\e>0$, we say that a continuous map
\[
t\mapsto \ue(t,\cdot)\in \lnorm 1(0,L), \quad \forall t\in (0,T)
\]
is an \emph{$\e$-approximate front tracking solution} to system \eqref{e:intro} if
\begin{enumerate}[(1)]
\item $\ue=\ue(t,x)\in \brz $ for all $(t,x)\in \bar \D:=\{0\le t\le T,\ 0\le x\le L\}$ as a function of two variables, and is piecewise constant with discontinuities occurring along finitely many straight lines with non-zero slope in the domain $\bar \D$. Jumps can be of two types: physical fronts (contact discontinuities) and non-physical fronts, denoted by $\P$ and $\NP$, respectively.
\item Along each physical front $x=\xat\ (\alpha \in \P)$, the left and right limits of $\ue(t,\cdot)$ on it are connected by
\begin{align*}
&\ur=R_{k_{\alpha}}(\sigma_{\alpha})[\ul], \quad \text{if $ k_{\alpha}\in\{1,...,k,k+p+1,...,n\} $},\\
&\ur=\Psi_{k_{\alpha}}(\sigma_{\alpha,k},...,\sigma_{\alpha,1})[\ul], \quad \text{if $ k_{\alpha}= k+1 $},
\end{align*}
where $\ul:=\ue(t,\xat-)$, $\ur:=\ue(t,\xat+)$, and $\sigma_{\alpha}$ or $(\sigma_{\alpha,{p}},...,\sigma_{\alpha,1}) $ is the wave amplitude. Moreover, the speed of the front approximately satisfies the Rankine-Hugoniot relation, that is
\begin{equation}
\label{d:approx-shock}
|\dot x_\alpha-\lambda_{k_{\alpha}}(\ul)|=\oo \e.
\end{equation}
\item All non-physical fronts $x=\xat\ (\alpha\in \NP)$ have the constant speed $\dot x_{\alpha}\equiv \hat \lambda$ with either $\hat \lambda>\sup_{\substack{u\in \brz\\1\le i\le n}}|\lambda_i(u)|$ or $0<\hat \lambda<c$, where $ c $ is given by \eqref{h:nonzero-char}. Moreover, the total strength of all non-physical waves in $\ue(t,\cdot)$ is uniformly bounded by $\e$, namely,
\[
\sum_{\alpha\in\NP}|\ue(t,\xa+)-\ue(t,\xa-)|\le \e,\qquad \forall t\in (0,T).
\]
\end{enumerate}
	
Moreover, if the initial and boundary values of $\ue$ satisfy approximatively the initial and boundary conditions, that is,
\begin{equation*}
\|\ue(0,\cdot)-\bar u\|_{\lnorm 1(0,L)}\leq \e,
\end{equation*}
\begin{equation*}
\label{d:aprrox_boundary}
\|b_1\big(\ue(\cdot,0+)\big)-g_1\|_{\lnorm{1}(0,T)}\leq \e,\quad \|b_2\big(\ue(\cdot,L-)\big)-g_2\|_{\lnorm{1}(0,T)}\leq \e,
\end{equation*}
then $\ue=\ue(t,x)$ is called the $\e$-approximate front tracking solution to the initial-boundary value problem \eqref{eqt:ibvp}. For shortness in what follows, we will call the $\e$-approximate front tracking solution just as the \emph{$\e$-solution}.
\end{definition}

\vspace{6pt}
For any given $ T>0 $, any given initial-boundary data $ (\iu,g_1,g_2) $ and any given $ \e>0 $ small enough, if $ \lug $ is sufficiently small, we can construct an $ \e $-solution to problem \eqref{eqt:ibvp} on the domain $ \D $ via an algorithm given in \cite{Li_Yu-onesided2016}, such that for all $ \e$ small, the maps $t \mapsto \ue(t,\cdot)$ are uniformly Lipschitz continuous in $ \lnorm 1 $ norm with respect to $ t $ and $ \displaystyle{\tv{\ue(t,\cdot)}{0<x<L} }$ remain sufficiently small uniformly for all $ t\in (0,T) $. Moreover, the approximate stability holds for $ \ue $ on the triangle domains
\[
\L(x_1):=\left\{(t,x)\ |\ 0<t<\hat \tau_1(x_1),0<x< x_1(\hat \tau_1(x_1)-t)/{\hat \tau_1(x_1)}\right\}
\]
and
\[
\mathfrak{R}(x_0):=\left\{(t,x)\ |\ 0<t<\hat\tau_2(x_0),\ (L-x_0)t/\hat{\tau}_2(x_0)+x_0<x<L \right\}
\]
for any given $ x_1\in (0,L] $ and $ x_0\in [0,L) $, where
\begin{equation*}
\hat{\tau}_1(x_1)=x_1\min_{u\in \brz}\{|\lambda_1(u)|^{-1}\}\quad \text{and}\quad
\hat \tau_2(x_0)=(L-x_0) \min_{u\in \brz}\{\lambda_n(u)^{-1}\}.
\end{equation*}
By induction, we obtain the approximate stability of $ \e $-solutions on the domain $ \D $.

Now, fix a sequence  $\enu\searrow 0$ as $ \nu\to +\infty $. By Helly's Theorem \cite[Theorem 2.3]{Bressan2000}, we can extract a subsequence of $ \{\unu\} $ which converges to a limit function $u=u(t,x)$ in $\lnorm 1((0,T)\times(0,L))$. In fact, we have the following theorem.

\begin{proposition}
\label{t:entropy-solution}
For any fixed $T>0$, there exist positive constants $\delta$ and $ C(T) $ such that for every initial-boundary data $(\iu,g^u_1,g^u_2)$ with
\begin{equation*}
\Lambda(\iu,g^u_1,g^u_2) \le \delta,
\end{equation*}
where
\begin{equation*}
\label{def:lug}
\begin{split}
\lug:=\tv{\iu}{0<x<L}+&|\iu(0+)|+\sum_{i=1,2}\tv{g_i}{0<t<T}\\
+&|b_1(\iu(0+))-\go(0+)|+|b_2(\iu(L-))-\gt(0+)|,
\end{split}
\end{equation*}
such that problem \eqref{eqt:ibvp} associated with the initial-boundary data $ (\iu,g^u_1,g^u_2) $ admits a solution $ u=u(t,x) $ on the domain $ \D=\{0<t<T,\ 0<x<L\} $ as the limit of a sequence of $\e$-solutions, satisfying
\begin{equation}
\label{bd:tv}
\begin{split}
\tv{u(t,\cdot)}{0<x<L}\leq  C(T) \lug, \quad \forall t\in (0,T),\\
\|u(t,\cdot)-u(s,\cdot)\|_{\lnorm 1(0,L)}\leq C(T) |t-s|,\quad \forall t,s \in (0,T)
\end{split}
\end{equation}
and $u(t,x)\in \brz$ for a.e. $(t,x)\in \D  $.
	
Moreover, if  $ v=v(t,x) $ is a solution as the limit of a sequence of $\enu$-solutions of system \eqref{e:intro}, associated with the initial-boundary data $(\iv,g^v_1,g^v_2)$ with $ \Lambda(\iv,g^v_1,g^v_2)\le \delta $, then for any given $ x_0\in [0,L) $ and $ x_1\in (0,L] $, there exist a positive constant $ C $ independent of $ x_0 $ and $ x_1 $, such that
\begin{equation}
\label{est:stability-left-triangle}
\begin{split}
&\|u(t,\cdot)-v(t,\cdot)\|_{\lnorm 1(\L_t(x_1))} \\
\leq &C \left(\|\iu-\iv\|_{\lnorm 1(0,x_1)} +  \int^{t}_0 |g^u_1(s)-g^v_1(s))|ds \right),\quad\forall t\in [0,\hat{\tau}_1(x_1)],
\end{split}
\end{equation}
\begin{equation}
\label{est:stability-right-triangle}
\begin{split}
&\|u(t,\cdot)-v(t,\cdot)\|_{\lnorm 1(\mathfrak{R}_t(x_0))}\\
\leq & C \left(\|\iu(0,\cdot)-\iv(0,\cdot)\|_{\lnorm 1(x_0,L)} + \int^{t}_0 |g^u_2(s)-g^v_2(s)|ds \right),\quad \forall t\in[0,\hat\tau_2(x_0)]
\end{split}
\end{equation}
where
\[
\begin{split}
\L_t(x_1):=\left\{x\ |\ 0<x<x_1(\hat \tau_1(x_1)-t)/{\hat \tau_1(x_1)}\right\},\\
\mathfrak{R}_t(x_0):=\left\{x\ |\ (L-x_0)t/\hat{\tau}_2(x_0)+x_0<x<L\right\},
\end{split}
\]
and there exists a positive constant $ C(T) $ depending on time $ T $, such that
\begin{equation}
\label{e:stability}
\begin{split}
&\|u(t,\cdot)-v(t,\cdot)\|_{\lnorm 1(0,L)}\\
\leq & C(T)  \left(\|\iu-\iv\|_{\lnorm 1(0,L)} + \sum_{i=1,2}\int^{t}_{0} \big|g^u_i(s)-g^v_i(s)\big|ds \right),\quad \forall t\in (0,T).
\end{split}
\end{equation}
\end{proposition}
In particular, \eqref{e:stability} implies that the solution provided by Proposition \ref{t:entropy-solution} is independent of different choices of the convergent sequence of $\e$-solutions.

\begin{remark}
\label{r:entropy-solution}
Under the assumption that system \eqref{e:intro} possesses a convex entropy $ \zeta(u) $, the solution $ u=u(t,x) $ given by Proposition \ref{t:entropy-solution} is actually an entropy solution to the problem \eqref{eqt:ibvp} on the domain $ \D $, and the equality holds in \eqref{d:entropy} (see \cite[Section 7.4]{Bressan2000}).
\end{remark}

\begin{remark}
\label{r:det-domain}
As we mentioned in \cite{Li_Yu-onesided2016}, according to \eqref{est:stability-left-triangle} (resp. \eqref{est:stability-right-triangle}), the triangle domain $ \L(x_1) $ (resp. $ \mathfrak{R}(x_0) $) is the \emph{determinate domain} of the solution to one-sided initial-boundary value problem \eqref{e:intro} with the initial data on the interval $ (0,x_1) $ (resp. $ (x_0,L) $) and the boundary condition on $ x=0 $ (resp. $ x=L $).
In particular, let $ u=u(t,x) $ be the solution to problem \eqref{eqt:ibvp} on the domain $ \D $ given by Proposition \ref{t:entropy-solution}, with $ \lug $ sufficiently small. For any given $ x_0\in (0,L) $, if $ \iu\equiv 0 $ on $ (x_0,L) $ and $ g_2\equiv 0 $ on the interval $ (0,\hat{\tau}_2(x_0)) $, then $ u\equiv 0 $ on the domain $ \mathfrak{R}(x_0)\cap \D $.
\end{remark}

\vspace{6pt}

\subsection{Some further properties of $\e$-approximate front tracking solutions and solutions}
\label{s:further-properties}
The following lemma can be deduced from Lemma 2.10 in \cite{Li_Yu-onesided2016}.
\begin{lemma}
\label{l:approx-boundary}
Suppose $ \unu $ is $ \enu $-solutions to the mixed initial-boundary value problem \eqref{eqt:ibvp}. Then, up to a subsequence, as $ \nu\to \infty $ we have
\begin{eqnarray*}
\|\unu(\cdot,0+)-u(\cdot,0+)\|_{\lnorm{\infty}}\to 0,\\
\|\unu(\cdot,L-)-u(\cdot,L-)\|_{\lnorm{\infty}}\to 0.
\end{eqnarray*}
\end{lemma}

As we mentioned in Section \ref{s:structure}, in order to prove Theorem \ref{t:twoside-cl}-\ref{t:ts-lcl}, besides of the well-posedness of semi-global solution to problem \eqref{eqt:ibvp}, we also need prove that the $\e$-solution to the forward problem is also a $ \e $-solution in the leftward/rightward  /backward sense. In \cite{Li_Yu-onesided2016} and \cite{Yu-BCLD}, we prove that for the system \eqref{e:intro} satisfying Hypothesis (H1)-(H3) and the assumption that all negative (resp. positive) eigenvalue are linear degenerate, the $\e$-solution to the forward problem is also a $\e$-solution to \eqref{e:right-intro} in the rightward (resp. leftward) sense. Therefore, with additional Hypothesis (H4), we can obtain the following lemma.
\begin{lemma}
If $ \ue $ is an $ \e $-solution to the forward problem of \eqref{e:intro}, then $ \ue $ is also an $ \e $-solution to the system \eqref{e:right-intro} in the leftward/rightward sense . And vice versa.
\end{lemma}
\label{l:esflr}
This lemma immediately implies the equivalence between $ \e $-solution in the forward sense and backward sense of the system \eqref{e:intro}, that is
\begin{lemma}
\label{l:esfb}
If $ \ue $ is an $ \e $-solution to the forward problem of \eqref{e:intro}, then $ \ue $ is also an $ \e $-solution to the system \eqref{e:intro} in the backward sense. And vice versa.
\end{lemma}

Now, by passing to the limit, we obtain the following.
\begin{proposition}
\label{p:forward-rightward-solution}
Suppose $ u $ is a solution to problem \eqref{e:intro} associated with  some admissible final-boundary condition, then $ u $ is also a solution to \eqref{e:intro} in the forward sense. Moreover, $ u $ is also a solution to \eqref{e:right-intro} in the leftward/rightward sense.
\end{proposition}

\begin{remark}
Since system \eqref{e:right-intro} does not possess a convex entropy in general, even if system \eqref{e:intro} possesses a convex entropy, the solution in the left/rightward sense of system \eqref{e:right-intro} is not necessary to be an entropy solution, but it gives no influence to our consideration and results.
\end{remark}

Applying the same argument in \cite{Li_Yu-onesided2016} (or \cite{Yu-BCLD}), we can obtain the following proposition in which the initial-boundary condition is involved.
\begin{proposition}
\label{p:triangle-rightward}
Suppose that $ u=u(t,x) $ is a forward solution to problem \eqref{eqt:ibvp} on the domain $ \{0<t<T_1,\ 0<x<L\} $ with $ \displaystyle T_1\geq L\max_{u\in \brz}\frac{1}{|\lambda_m(u)|} $ given by Proposition \ref{t:entropy-solution}. Then on the triangle domain $ \{0<t<T_1,\ 0<x<L(T_1-t)/T_1\} $, $ u $ coincides with the leftward (resp. rightward) solution $ \tilde u $ to system \eqref{e:right-intro} given by Proposition \ref{t:entropy-solution}, associated with the initial condition
\begin{equation*}
\label{ic:artifical-rightward}
x=0:\ \tu=u(\cdot,0+)
\end{equation*}
and the following boundary condition reduced from the original initial data $ \iu $:
\begin{equation*}
\label{bc:artifical-rightward}
t=0:\ \tilde b_2(\tu)=\tilde b_2(\iu)\quad (\text{resp.}\ t=0:\ \tilde b_1(\tu)=\tilde b_1(\iu)),
\end{equation*}
where $ \tilde b_2 \in \C1(B_r(0)$; $\rn{m})$ (resp.  $ \tilde b_1  \in \C1(B_r(0)$; $\rn{n-m})$) is arbitrarily given, satisfying the same assumption \eqref{h:bc} for $ b_2 $ (resp. $ b_1 $).
	
The similar results hold for a solution $ u =u(t,x)$ to problem \eqref{eqt:ibvp} in backward sense.	
\end{proposition}

In the proof of Theorem \ref{t:twoside-cl}, we need to consider two solutions obtained by solving the system \eqref{e:right-intro} leftward from $x=L$ and rightward from $ x=0$, respectively. Then the combination of these two solutions should be proven to be a solution to the system \eqref{e:intro} in the forward sense. In fact, we have
\begin{proposition}
\label{p:cslrf}
Suppose that $ u=u_l(t,x) $ ($u=u_r(t,x)$) is a solution to problem \eqref{eqt:ibvp} in the leftward (resp. rightward) sense on the domain $ \D_l:=\{0<t<T,\ 0<x<L/2\} $ (resp. $ \D_r:= \{0<t<T, L/2<t<L\} $) with initial condition
\begin{equation*}
x=L/2:\ u=a(t),\quad 0<t<T
\end{equation*}
for some function $ a $ and some suitable boundary conditions. Let
\begin{equation*}
\label{e:cslr}
u(t,x)=
\begin{cases}
u_l(t,x), & (t,x)\in \D_l,\\
u_r(t,x) & (t,x)\in \D_r.
\end{cases}
\end{equation*}
Then $ u=u(t,x) $ is a solution to system \eqref{e:intro}	in the forward sense.
\end{proposition}

\begin{proof}
Suppose $ u_l $ (resp. $ u_r $) is the limit of a sequence of $ \enu $-solutions $ \unu_l $ (resp. $ \unu_r $). For each $ \nu\ge 1 $, we define the function
\begin{equation*}
\unu(t,x)=
\begin{cases}
\unu_l(t,x) & (t,x)\in \D_l,\\
\unu_r(t,x) & (t,x)\in \D_r.
\end{cases}
\end{equation*}
In the interior of $ \D_l $ (resp. $ \D_r $), $ \unu_l $ (resp. $\unu_r  $) is an $ \enu $-solution to the system \eqref{e:intro} in the forward sense. It suffices to clarify the situation near the segment $ \mathfrak{S}:={0<t<T}\times\{x=L\} $. In a small leftward (resp. rightward) neighborhood of $\mathfrak{S} $, $ \unu_l $ (resp. $ \unu_r $) is obtained by an approximate Riemann solver at each jump points of $ a^{\nu}=a^{\nu}(t) $ which is a piecewise constant approximation of $ a $. By the finite speed of wave propagation, we know that there is no jump discontinuity for $ \unu $ in a small neighborhood of the segment $ \mathfrak{S}$ except for those jump point of $ a^{\nu} $ (see Figure \ref{f:fm}). Therefore, $ \unu $ is an $ \enu $-solutions to system \eqref{e:intro} in the forward sense. By passing to the limit, we obtain the solution $ u=u(t,x) $ as the limit of sequence $ \unu=\unu(t,x) $ to the system \eqref{e:intro} in the forward sense.
\end{proof}

\begin{figure}[H]
	\centering
	\includegraphics[width=0.4\textwidth]{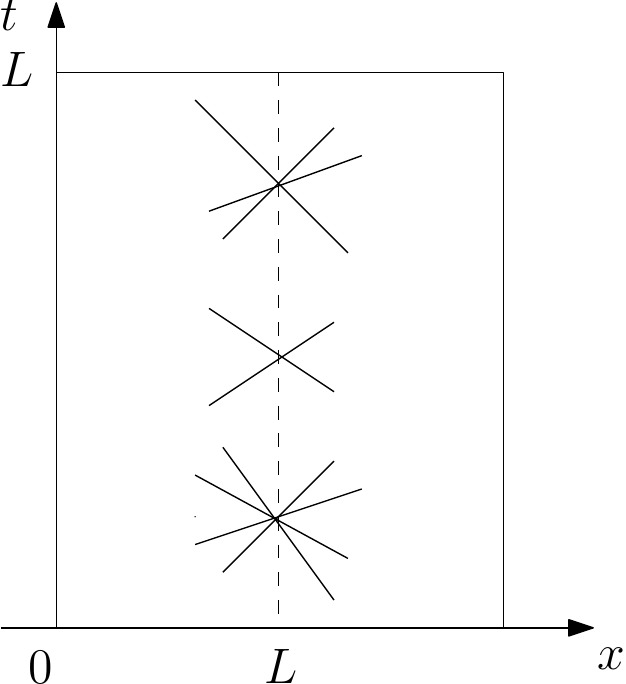}
	\caption{Fonts generated by approximate Riemann solver on the segment $ x=L $.}
	\label{f:fm}
\end{figure}

\section{Local exact boundary controllability}
\label{s:lebc}
Now we are ready to apply the well-posedness of semi-global solutions constructed as the limit of $\e$-solutions to prove Theorem \ref{t:twoside-cl}-\ref{t:ts-lcl}, namely, to realize the local exact boundary controllability for a class of general linear degenerate hyperbolic systems of conservation laws with eigenvalue of constant multiplicity.


\subsection{Two sided boundary control---proof of Theorem \ref{t:twoside-cl}}
In order to get Theorem \ref{t:twoside-cl}, it suffices to establish the following
\begin{lemma}
\label{l:twoside-bc}
Under the same assumptions of Theorem \ref{t:twoside-cl}. Let $ T>0 $ satisfies \eqref{d:tcl-two}. For any given initial data $\iu$ and final data $ u_1 $ with $\displaystyle{\tv{\iu}{0<x<L}+|\iu(0+)|}$ and $\displaystyle{\tv{u_1}{0<x<L}+|u_1(0+)|}$ sufficiently small, system \eqref{e:intro} admits a solution $ u=u(t,x) $ on the domain $\D$ with small $ \displaystyle \tv{u(\cdot,L-)}{0<t<T} +u(0,L-)$, satisfying simultaneously the initial condition
\begin{equation}
\label{e:ic}
t=0: \ u=\iu(x),\quad 0<x<L
\end{equation}
and the final condition \eqref{e:fc}.
\end{lemma}

In fact, let $ u=u(t,x) $ be a solution given by Lemma \ref{l:twoside-bc}. Taking the boundary control as
\[
g_1(t):=b_1(u(t,0+)),\quad g_2(t):=b_2(u(t,L-)),\quad \forall t\in(0,T),
\]
which has small amplitude and total variation, 
we obtain the local exact two-sided boundary controllability desired by Theorem \ref{t:twoside-cl}.

\begin{proof}[Proof of Lemma \ref{l:twoside-bc}]
	Noting \eqref{d:tcl-two}, for $ r>0 $ sufficiently small we have
	\begin{equation}
	\label{e:time-control-brz}
	T>L\max_{u\in \brz}\left\{\frac{1}{|\lambda_m(u)|}+\frac{1}{\lambda_{m+1}(u)}\right\}.
	\end{equation}
	Let
	\begin{equation}
		\label{d:To}
		T_1:=L\max_{u\in \brz}\left\{\frac{1}{|\lambda_m(u)|},{1\over \la_{m+1}(u) }\right\}.
	\end{equation}
Step 1.
Choosing any functions $g'_1:(0,T_1)\to \rn{n-m}$ and $ g'_2:(0,T_1)\to \rn{m}$ with $\displaystyle\tv{g'_1}{0< t< T_1}+|g'_1(0+)|$ and $\displaystyle\tv{g'_2}{0< t< T_1}+|g'_2(0+)|$ sufficiently small, we consider the forward problem of \eqref{e:intro} with the initial condition \eqref{e:ic} and the following artificial boundary conditions:
\begin{equation*}
\begin{cases}
x=0: & b_1(u)=g'_1(t),\\
x=L: & b_2(u)=g'_2(t).
\end{cases}
\end{equation*}
By Proposition \ref{t:entropy-solution} there exists a unique solution $u_f=u_f(t,x)$ as the limit of a sequence of $\enu$-solutions $\unu_f=\unu_f(t,x)$ on the domain $R_f=\{0<t<T_1,0<x<L\}$ with $\displaystyle\sup_{T_1-T<t<T}\tv{u_f(t,\cdot)}{0< x< L}+\tv{u_f(\cdot,0+)}{0< t< T_1}+\tv{u_f(\cdot,L-)}{0< t< T_1}$ sufficiently small and $u_f(t,x)\in \brz$ for a.e. $ (t,x)\in R_f $.
	
\vspace{6pt}

Step 2. We consider the backward mixed initial-boundary value problem of \eqref{e:intro} with the final condition \eqref{e:fc} and the artificial boundary  conditions
\begin{align*}
x=0:\quad  & l_r(u)u=g''_r(t) \qquad (r=1,...,m),\\
x=L:\quad & l_s(u)u=g''_s(t) \qquad (s=m+1,...,n),
\end{align*}
where $ g''_i \ (i=1,...,n) $ are any given functions of $ t $ with $ \displaystyle \tv{g''_i}{T-T_1<t<T}+|g''_i(0+)| $ sufficiently small. By Proposition \ref{t:entropy-solution}, there exists a solution $ u=u_b(t,x) $ on the domain
\[
R_b=\{ T-T_1<t<T,\ 0<x<L \}.
\]
with $\displaystyle \sup_{T_1-T<t<T}\tv{u_b(t,\cdot)}{0< x< L}+\tv{u_f(\cdot,0+)}{T-T_1< t< T}+\tv{u_b(\cdot,L-)}{T_1-T< t< T}$ sufficiently small and $u_b(t,x)\in \brz$ for a.e. $ (t,x)\in R_b $.

\vspace{6pt}	
	
Step 3. Noting \eqref{e:time-control-brz}-\eqref{d:To}, we can find a function $ a(t):(0,T)\to \rn{n} $ with $ \displaystyle \tv{a}{0<t<T} + |a(0+)|$ sufficiently small, such that
\[
a(t)=
\begin{cases}
u_f(t,L/2), & 0<t<T_1,\\
u_b(t,L/2), & T-T_1<t<T.
\end{cases}
\]
	
Now we change the role of variables $t$ and $x$ and consider the leftward problem for system \eqref{e:right-intro}
with the final condition
\begin{equation}
\label{ic:rightward}
x=L/2:\ u=a(t),\quad 0<t<T
\end{equation}
and  the following boundary conditions reduced from the initial data $u=\iu$ and the finial data $u_1$:
\begin{eqnarray}
t=0: && l_r(u) u=l_r(\iu)\iu, \quad r=1,...,m,\quad 0<x<L/2, \label{bc:art-left-0}
\\
t=T: && l_{s}(u) u=l_s(u_1)u_1, \quad s=m+1,...,n,\quad 0<x<L/2, \label{bc:art-left-T}
\end{eqnarray}
where $l_i(u)\ (i=1,...,n)$ are the left eigenvectors of $(DH(u)^{-1}DG(u) $, equivalently, the left eigenvectors of $ (DG(u))^{-1}DH(u) $. A direct computation shows that this boundary condition satisfies the assumption \eqref{h:bc}.
	
	
Still by Proposition \ref{t:entropy-solution}, the leftward problem admits a solution $u_l=u_l(t,x)$ on the domain $ \{ 0<t<T,\ 0<x<L/2 \} $ as the limit of a sequence of $ \enu $-solutions $ \unu_l $. 
\vspace{6pt}	

Step 4. Similarly, the rightward mixed initial-boundary value problem for system \eqref{e:right-intro} with the initial condition \eqref{ic:rightward} and the following reduced boundary conditions:
\begin{eqnarray}
t=0: && l_s(u) u=l_s(\iu)\iu, \quad s=m+1,...n,\quad L/2<x<L \label{bc:art-right-0}
\\
t=T: && l_{r}(u) u=l_r(u_1)u_1, \quad r=1,...,m,\quad L/2<x<L \label{bc:art-right-T}
\end{eqnarray}	
admits a solution $ u_r=u_r(t,x) $ on the domain
\[
R_r(T)=\left\{ 0<t<T,\ L/2<x<L  \right\},
\]
as the limit of a sequence of $ \enu $-solutions $ \unu_r $.
 	
Step 5. Let
\begin{equation}
\label{d:forwardu}
u(t,x)=
\begin{cases}
u_l(t,x), & (t,x)\in R_l(T),\\
u_r(t,x) & (t,x)\in R_r(T).
\end{cases}
\end{equation}
By Proposition \ref{p:cslrf}, $ u=u(t,x) $ is a solution to system \eqref{e:intro}.

Now it remains to show that $u$ verifies the initial condition \eqref{e:ic} and the final condition \eqref{e:fc}.

By Proposition \ref{p:forward-rightward-solution}, both $ u_f $ and $ u_l $ (resp. $ u_r $) are solutions to system \eqref{e:right-intro} in the leftward (resp. rightward) sense, with the same final (resp. initial) condition
\[
x=L/2:\quad u=a(t), \quad 0<t<T_1
\]
and the same boundary condition \eqref{bc:art-left-0} (resp. \eqref{bc:art-right-0}). Then by Proposition \ref{p:triangle-rightward} and Remark \ref{r:det-domain} for the rightward problem, and noting \eqref{d:To},  $u_f$ coincides with $ u_l $ (resp. $ u_r $) on the triangle domain
\begin{align*}
\{0<  t < 2T_1x/L,\  0< x <L/2\}. \\
\left( \text{resp.}\ \left\{ 0<t<2T_1(L-x)/L,\ L/2<x<L \right\}\right)
\end{align*}
Since $ u=u_f(t,x) $ satisfies the initial condition \eqref{e:ic}, this implies that $ u=u(t,x) $ given by \eqref{d:forwardu} verifies \eqref{e:ic}.

Similarly, $ u=u(t,x) $ verifies \eqref{e:fc}. Thus $ u=u(t,x) $ is a desired solution and the proof of Lemma \ref{l:twoside-bc} is complete.
\end{proof}

\subsection{One-sided boundary control---proof of Theorem \ref{t:oneside-cl}}

In order to get Theorem \ref{t:oneside-cl}, it suffices to establish the following
\begin{lemma}
\label{lemma:oneside-bc}
Under the same assumptions of Theorem \ref{t:oneside-cl}. Let $ T>0$ satisfy \eqref{d:tcl-one}. For any given initial data $\iu$ and boundary data $ g_1 $ with $\displaystyle{\tv{\iu}{0<x<L}+|\iu(0+)|}$ and $ \displaystyle \tv{g_1}{0<t<T}+|g(0+)|$ sufficiently small, system \eqref{e:intro} together with the boundary condition
\begin{equation}
\label{bc:reduce-oneside-IBVP}
x=0:\ b_1(u)=g_1(t), \qquad t\in (0,T)
\end{equation}
admits a solution $ u=u(t,x) $ on the domain $\{\ 0< t< T,\ 0< x< L\}$ with small $ \displaystyle \tv{u(\cdot,L-)}{0<t<T} +u(0,L-)$, satisfying simultaneously the initial condition \eqref{e:ic} and the final condition \eqref{e:fc}.
\end{lemma}

In fact, let $ u=u(t,x) $ be a solution given by Lemma \ref{lemma:oneside-bc}. Taking the boundary control as
\[
g_2(t):=b_2(u(t,L-)),\quad \forall t\in(0,T),
\]
which has small amplitude and total variation, 
we obtain the local exact one-sided boundary controllability desired by Theorem \ref{t:oneside-cl}.

\begin{proof}[Proof of Lemma \ref{lemma:oneside-bc}]
Noting \eqref{d:tcl-one}, for $ r>0 $ sufficiently small we have
\begin{equation}
\label{e:time-control-brz}
T>L\cdot\max_{u\in \brz}\left\{\frac{1}{|\lambda_m(u)|}+\frac{1}{\lambda_{m+1}(u)}\right\}.
\end{equation}
Step 1. Let
\begin{equation}
\label{d:To}
T_1:=L\cdot\max_{u\in \brz}\frac{1}{|\lambda_m(u)|}.
\end{equation}
Choosing an artificial function $g_f$ with $\displaystyle\tv{g_f}{0< t< T_1}+|g_f(0+)|$ sufficiently small, we consider the forward problem of \eqref{e:intro} with the initial condition \eqref{e:ic} and the following boundary conditions:
\begin{equation*}
\begin{cases}
x=0: & b_1(u)=g_1(t),\\
x=L: & b_2(u)=g_f(t),
\end{cases}
\qquad t\in (0,T_1).
\end{equation*}
By Proposition \ref{t:entropy-solution} there exists a unique solution $u_f=u_f(t,x)$ as the limit of a sequence of $\enu$-solutions $\unu_f=\unu_f(t,x)$ on the domain $\{0<t<T_1,0<x<L\}$ with $\displaystyle\tv{u_f(t,\cdot)}{0< x< L}+\tv{u_f(\cdot,0+)}{0< t< T_1}+\tv{u_f(\cdot,L-)}{0< t< T_1}$ sufficiently small and $u_f(t,x)\in \brz$.
	
\vspace{6pt}

Step 2. Let
\[
T_2=L\max_{u\in \brz} {1\over \la_{m+1}(u)}.
\]
Noting \eqref{h:os-b_1}, without loss of generality, we assume that
\begin{equation*}
\label{h:b1barm}
\det \big[ Db_1(u)\cdot r_1(u)\ |\cdots |\ Db_1(u)\cdot r_{\bar{m}}(u) \big]\ne 0.
\end{equation*}
By Proposition \ref{t:entropy-solution}, the backward initial-boundary value problem \eqref{e:intro} admits a solution $ u=u_b(t,x) $ on the domain
\[
R_b=\{ T-T_2<t<T,\ 0<x<L \},
\]
satisfying the final condition \eqref{e:fc}, the boundary condition \eqref{bc:intro-a} and the following artificial boundary condition
\begin{align*}
x=0:\quad & l_p(u)u=g_p(t) \quad p=\bar m,...,m,\\
x=L:\quad & l_s(u)u=g_s(t) \quad s=m+1,...,n,
\end{align*}
where $ g_i:(T-T_2,T)\to \R\ (i=\bar m,...,n) $ are any given functions of $ t $ with $ \displaystyle \tv{g_i}{T-T_2<t<T}+|g_i(0+)|  $ sufficiently small.

Step 3. Let
\[
a(t)=\begin{cases}
u_f(t,0+) & 0<t<T_1,\\
u_b(t,0+) & T-T_2<t<T	.	
\end{cases}
\]
Obviously, $ a(t)\in \brz $  with sufficiently small total variation, and $ u=a(t) $ satisfies the boundary condition \eqref{bc:reduce-oneside-IBVP} at $ x=0 $ on the whole time interval $ (0,T) $.
	
Now we change the role of variables $t$ and $x$ and consider the rightward problem for system
\begin{equation*}
\px G(u)+\pt H(u)=0, \qquad 0<x<L,\ 0<t<T
\end{equation*}
with the initial condition
\begin{equation*}
x=0:\ u=a(t),\quad 0<t<T
\end{equation*}
and the following boundary conditions reduced from the initial state $u=\iu$ and the finial state $u=0$:
\begin{eqnarray*}
t=0: && l_s(u) u=l_s(\iu)\iu, \quad s=m+1,...n,\nonumber 
\\
t=T: && l_{r}(u) u=l_r(u_1)u_1, \quad r=1,...,m, \label{bc:artifical-right}
\end{eqnarray*}
where $l_i(u)\ (i=1,...,n)$ are the left eigenvectors of $(DH(u)^{-1}DG(u) $, equivalently, the left eigenvectors of $ (DG(u))^{-1}DH(u) $. A direct computation shows that this boundary condition satisfies the assumption \eqref{h:bc}.
	
	
Still by Proposition \ref{t:entropy-solution}, the rightward problem admits a solution $u=u(t,x)$ on the domain $ \{ 0<t<T,\ 0<x<L \} $ as the limit of a sequence of $ \enu $-solutions $ \unu $. By Proposition \ref{p:forward-rightward-solution}, $ u $ is also a solution of system \eqref{e:intro} in the forward sense on $ \{0<t<T,\ 0<x<L\} $. Since $ u(t,0)=a(t) $ for a.e. $ t\in (0,T) $, we have
\[
b_1(u(t,0+))=g_1(t), \qquad \text{a.e.}\ t\in (0,T).
\]	
	
Step 4. Now it remains to show that $u$ verifies the initial condition \eqref{e:ic} and the final condition \eqref{e:fc}.
	
By Proposition \ref{p:forward-rightward-solution}, both $ u_f $ and $ u $ are solutions in the rightward sense. Then by Proposition \ref{p:triangle-rightward} and Remark \ref{r:det-domain} for the rightward problem, and noting \eqref{d:To},  $u_f$ coincides with $ u $ on the triangle domain $\{0\leq  t \le T_1,\  0\leq x \le L(T_1-t)/T_1\}$. This implies \eqref{e:ic}. Similarly, we can get \eqref{e:fc}.
	
Thus $ u=u(t,x) $ is a desired solution and the proof of Lemma \ref{lemma:oneside-bc} is complete.
\end{proof}	

\subsection{Two-sided boundary control with less controls---proof of Theorem \ref{t:ts-lcl}}

In order to get Theorem \ref{t:ts-lcl}, it suffices to establish the following
\begin{lemma}
\label{l:twoside-lesscontrol}
Under the same assumptions of Theorem \ref{t:ts-lcl}, for any given initial data $\iu$ any given $ \tilde g_2:(0,T)\to \rn{\bar m} $ with $\displaystyle{\tv{\iu}{0<x<L}+|\iu(0+)|}$ and $ \displaystyle \tv{\tilde g_1}{0<t<T}  $ sufficiently small, system \eqref{e:intro} together with the boundary condition
\begin{equation}
\label{bc:reduce-tl}
x=L:\ \tilde b_2(u)=\tilde{ g}_2, \qquad t\in (0,T)
\end{equation}
admits a solution $ u=u(t,x) $ on the domain $\{\ 0< t< T,\ 0< x< L\}$ with small $ \displaystyle \tv{u(\cdot,L-)}{0<t<T} +u(0,L-)$, satisfying simultaneously the initial condition \eqref{e:ic} and the final condition \eqref{e:fc}.
\end{lemma}

In fact, let $ u=u(t,x) $ be a solution given by Lemma \ref{lemma:oneside-bc}. Taking the boundary control as
\begin{align*}
\hat g_1(t):= b_1(u(t,0+)),\quad
\hat g_2(t):=\hat b_2(u(t,L-)),\quad \forall t\in(0,T),
\end{align*}
where $ \hat b_1 $ is the vector function consists of the last $m-\bar m$ components of $ b_1 $, 
we obtain the local exact two-sided boundary controllability with less controls desired by Theorem \ref{t:oneside-cl}.

\begin{proof}[Proof of Lemma \ref{l:twoside-lesscontrol}]
Noting \eqref{d:tcl-one}, for $ r>0 $ sufficiently small, \eqref{e:time-control-brz} holds.

Step 1. Let
\begin{equation}
\label{d:to}
T_1:=L\cdot\max_{u\in \brz}\frac{1}{|\lambda_m(u)|}.
\end{equation}
Choosing artificial boundary data $ g'_1 $ and $\hat g'_2:(0,T_1)\to \rn{m-\bar m}$ and $  $ with $\displaystyle\tv{g'_1}{0< t< T_1}+|g'_1(0+)|$ and $ \displaystyle \tv{\hat{g}'_2}{0<t<T_1}+|\hat{g}'_2(0+)|$ sufficiently small, we consider the forward problem of \eqref{e:intro} with the initial condition \eqref{e:ic}, the boundary condition \eqref{bc:reduce－tl} on $ x=L $ and the following boundary conditions:
\begin{eqnarray*}
x=0: && b_1(u)=g'_1(t),\\
x=L: && \hat b_2(u)=\hat g'_2(t).
\end{eqnarray*}
By Proposition \ref{t:entropy-solution} there exists a unique solution $u_f=u_f(t,x)$ as the limit of a sequence of $\enu$-solutions $\unu_f=\unu_f(t,x)$ on the domain $R_f=\{0<t<T_1,0<x<L\}$ with $\displaystyle\tv{u_f(t,\cdot)}{0< x< L}+\tv{u_f(\cdot,0+)}{0< t< T_1}+\tv{u_f(\cdot,L-)}{0< t< T_1}$ sufficiently small and $u_f(t,x)\in \brz$ for a.e. $ (t,x)\in R_f $.
	
\vspace{6pt}

Step 2.  Let
\begin{equation}
\label{d:tt}
T_2=L\max_{u\in\brz}\frac{1}{\la{_{m+1}(u)}}.
\end{equation}
Choose any functions $ g''_1 $ with $ \displaystyle \tv{g''_1}{T-T_2<t<T}+|g''_1(0+)|$ sufficiently small. By Proposition \ref{t:entropy-solution}, there exist a solution to the mixed initial boundary problem \eqref{e:intro} on the domain
\[
R_b=\{ T-T_2<t<T,\ 0<x<L\},
\]
with the final condition \eqref{e:fc}, the boundary condition \eqref{bc:reduce-tl} and the artificial boundary condition
\[
x=0:\quad b_1(u)=g''_1(t).
\]

Step 3. Noting \eqref{d:tcl-one}, \eqref{d:to} and \eqref{d:tt}, we can find a function $ a(t) $ with $ \displaystyle\tv{a}{0<t<T}+|a(0+) $ sufficiently small, such that
\[
a(t)=
\begin{cases}
u_f(t,L-) & 0<t<T_1,\\
u_b(t,L-) & T-T_2<t<T,	.	
\end{cases}
\]
and $ u=a(t) $ satisfies the boundary condition \eqref{bc:reduce-tl} at $ x=L $ on the whole time interval $ (0,T) $.
	
Now we change the role of variables $t$ and $x$ and consider the leftward problem for system
\begin{equation*}
\px G(u)+\pt H(u)=0, \qquad 0<x<L,\ 0<t<T
\end{equation*}	
with the final condition
\begin{equation*}
x=L:\ u=a(t),\quad 0<t<T
\end{equation*}
and  the following boundary conditions reduced from the initial state $u=\iu$ and the finial state $u=0$:
\begin{eqnarray*}
t=0: && l_r(u) u=l_r(\iu)\iu, \quad s=1,...m,\nonumber 
\\
t=T: && l_{s}(u) u=l_s(u_1)u_1, \quad r=m+1,...,n, \label{bc:artifical-right}
\end{eqnarray*}
where $l_i(u)\ (i=1,...,n)$ are the left eigenvectors of $(DH(u)^{-1}DG(u) $, equivalently, the left eigenvectors of $ (DG(u))^{-1}DH(u) $. A direct computation shows that this boundary condition satisfies the assumption \eqref{h:bc}.
	
	
Still by Proposition \ref{t:entropy-solution}, the leftward problem admits a solution $u=u(t,x)$ on the domain $ \{ 0<t<T,\ 0<x<L \} $ as the limit of a sequence of $ \enu $-solutions $ \unu $. By Proposition \ref{p:forward-rightward-solution}, $ u $ is also a solution of system \eqref{e:intro} in the forward sense on $ \{0<t<T,\ 0<x<L\} $.
	
Step 4. Now it remains to show that $u$ verifies the initial condition \eqref{e:ic} and the final condition \eqref{e:fc}.
	
By Proposition \ref{p:forward-rightward-solution}, both $ u_f $ and $ u $ are solutions in the leftward sense. Then by Proposition \ref{p:triangle-rightward} and Remark \ref{r:det-domain} for the leftward problem, and noting \eqref{d:To},  $u_f$ coincides with $ u $ on the triangle domain $\{0\leq  t \le T_1,\  0\leq x \le L(t-T_1)/T_1\}$. This implies \eqref{e:ic}. Similarly, we can get \eqref{e:fc}.
	
Thus $ u=u(t,x) $ is a desired solution and the proof of Lemma \ref{lemma:oneside-bc} is complete.
\end{proof}

\bibliographystyle{siam}
\bibliography{boundary_control_LD}

\begin{thebibliography}{10}

\bibitem{AnconaCoclite2005}
{\sc F.~Ancona and G.~M. Coclite}, {\em On the attainable set for temple class
  systems with boundary controls}, SIAM Journal on Control and Optimization, 43
  (2005), pp.~2166--2190.

\bibitem{Bressan2000}
{\sc A.~Bressan}, {\em {Hyperbolic Systems of Conservation Laws: The
  One-dimensional Cauchy Problem}}, Oxford lecture series in mathematics and
  its applications, Oxford University Press, USA, 2000.

\bibitem{Frei-LD}
{\sc H.~Freist{\"u}hler}, {\em Linear degeneracy and shock waves},
  Mathematische Zeitschrift, 207, pp.~583--596.

\bibitem{Glass2007}
{\sc O.~Glass}, {\em On the controllability of the {1-D} isentropic euler
  equation}, Journal of the European Mathematical Society, 9 (2007), pp.~427 --
  486.

\bibitem{Glass2014}
\leavevmode\vrule height 2pt depth -1.6pt width 23pt, {\em On the
  controllability of the non-isentropic {1-D} euler equation}, Journal of
  Differential Equations, 257 (2014), pp.~638 -- 719.

\bibitem{Kong-control2002}
{\sc D.-X. Kong}, {\em Global exact boundary controllability of a class of
  quasilinear hyperbolic systems of conservation laws}, Systems and Control
  Letters, 47 (2002), pp.~287 -- 298.

\bibitem{Kong-control2005}
{\sc D.-X. Kong and H.~Yao}, {\em Global exact boundary controllability of a
  class of quasilinear hyperbolic systems of conservation laws ii}, SIAM
  Journal on Control and Optimization, 44 (2005), pp.~140--158.

\bibitem{Lax1987}
{\sc P.~D. Lax}, {\em {Hyperbolic Systems of Conservation Laws and The
  Mathematical Theory of Shock Waves}}, CBMS-NSF Regional Conference Series in
  Applied Mathematics, Society for Industrial and Applied Mathematics, USA,
  January 1987.

\bibitem{Li_controllability-book}
{\sc T.~Li}, {\em {Controllability and Observability for Quasilinear Hyperbolic
  Systems}}, vol.~3 of AIMS Series on Applied Mathematics, AIMS \& Higher
  Education Press, 2010.

\bibitem{Li_Yu-onesided2016}
{\sc T.~Li and L.~Yu}, {\em One-sided exact boundary null controllability of
  entropy solutions to a class of hyperbolic systems of conservation laws}.
\newblock To appear in Journal de Math\'{e}matiques Pures et Appliqu\'{e}es,
  2016.

\bibitem{Li-Yu_boundary-value}
{\sc T.~Li and W.~Yu}, {\em {Boundary Value Problems for Quasilinear Hyperbolic
  Systems}}, Mathematics Series V, Duke University, 1985.

\bibitem{Yu-BCLD}
{\sc L.~Yu}, {\em A note on "one-sided exact boundary null controllability of
  entropy solutions to a class of hyperbolic systems of conservation laws"}.
\newblock Preprint.

\end{thebibliography}

\end{document}